\documentclass[reqno]{amsart}
\usepackage{amssymb, amsthm, amsmath, xypic}
\usepackage[pagewise]{lineno}
\usepackage{color}

\calclayout

\theoremstyle{definition}
\newtheorem{theorem}{Theorem}[section]

\newtheorem{proposition}[theorem]{Proposition}
\newtheorem{lemma}[theorem]{Lemma}
\newtheorem{corollary}[theorem]{Corollary}
\newtheorem{example}[theorem]{Example}

\newcommand{\Z}{\mathbb{Z}}
\newcommand{\N}{\mathbb{N}}
\renewcommand{\bar}{\overline}

\DeclareMathOperator{\Perm}{\mbox{Perm}}
\DeclareMathOperator{\Hol}{\mbox{Hol}}
\DeclareMathOperator{\Aut}{\mbox{Aut}}

\begin{document}

\title[Bicyclic biskew braces]{Bicyclic biskew braces}

\author{Hal Simpson}
\address{School of Computer Science and Mathematics \\ Keele University \\ Staffordshire \\ ST5 5BG \\ UK}
\curraddr{School of Mathematics \\ University of Leeds \\Leeds \\ LS2 9JT \\ UK}
\email{H.P.Simpson@leeds.ac.uk}

\author{Paul J. Truman}
\address{School of Computer Science and Mathematics \\ Keele University \\ Staffordshire \\ ST5 5BG \\ UK}
\email{P.J.Truman@keele.ac.uk}



\begin{abstract}
We study finite braces whose additive and multiplicative groups are both cyclic. We reinterpret the classification of these braces from the perspective of regular subgroups of permutation groups, and identify which of them are biskew braces.
\end{abstract}

\maketitle

\section{Introduction} \label{sec_introduction}

Braces are generalizations of radical rings introduced by Rump \cite{Ru07a} as an algebraic framework for studying involutive nondegenerate solutions of the set-theoretic Yang-Baxter equation. A (left) \textit{brace} is a triple $ (B,+,\circ) $ in which $ (B,+) $ is an abelian group, $ (B,\circ) $ is a group, and the \textit{brace relation} 
\begin{equation} \label{eqn_brace_relation}
x \circ (y + z) = (x \circ y) - x + (x \circ z) 
\end{equation}
is satisfied for all $ x,y,z \in B $. We call $ (B,+) $ the \textit{additive group}, and $ (B,\circ) $ the \textit{multiplicative group}, of the brace $ B = (B,+,\circ) $. It follows quickly from \eqref{eqn_brace_relation} that $ (B,+) $ and $ (B,\circ) $ share a common identity element, which we denote by $ e $. However, the inverse of an element $ x \in B $ with respect to $ + $ (denoted $ -x $) need not coincide with its inverse with respect to $ \circ $ (denoted $ \bar{x} $). We will study only \textit{finite} braces (that is, those whose underlying set is finite). 

Numerous generalizations and variants of braces have been developed, the most impactful of which is Guarnieri and Vendramin's notion of a \textit{skew brace} \cite{GV17}: these are defined in the same way as braces, but without the assumption that the additive group is abelian (thus Rump's braces are skew braces with abelian additive group). Skew braces yield nondegenerate solutions of the set-theoretic Yang-Baxter equation that are bijective but not necessarily involutive. 

An \textit{isomorphism} of skew braces is a bijective function between the underlying sets that respects both operations. A natural problem is to classify skew braces of given orders, or with given additive or multiplicative groups, up to isomorphism. For example: Rump classifies braces with cyclic additive group \cite{Ru07b}, Bachiller classifies braces of order $ p^{3} $ \cite{Ba15}, and Byott and Ferri enumerate braces whose multiplicative group is dihedral or generalized quaternion \cite{BF25}. Many methods for classifying skew braces exploit their connections with other algebraic objects, such as bijective $ 1 $-cocycles, regular subgroups of permutation groups, or Hopf-Galois structures on finite Galois extensions. In exploring the last of these Childs \cite{Ch19b} introduces \textit{biskew braces}: these are skew braces $ (B,+,\circ) $ such that $ (B,\circ,+) $ is also a skew brace. Properties of biskew braces are explored in \cite{ST23b}, and works such as \cite{Koc21a} describe methods for constructing biskew braces, but we are not aware of any work in the literature that classifies families of biskew braces.

In this paper we classify braces whose additive and multiplicative group are cyclic (\textit{bicyclic braces}), and identify which of these are biskew. Our main result is the following: 

\begin{theorem} \label{theorem_main}
The isomorphically distinct bicyclic braces of order $ n $ correspond bijectively with positive divisors $ w $ of $ n $ with the following properties:
\begin{itemize}
\item $ p \mid w $ for all primes $ p $ that divide $ n $;
\item if $ 4 \mid n $ then $ 4 \mid w $.
\end{itemize}
We may realize the bicyclic brace corresponding to such a divisor $ w $ by $ (\Z_{n},+,\circ) $, where $ + $ is the usual addition on $ \Z_{n} $ and $ \circ $ is defined by $ x \circ y = x + y + wxy $. A bicyclic brace of this form is biskew if and only if $ n \mid w^{2} $. 
\end{theorem}

The classification of bicyclic braces can be extracted from Rump's classification of braces with cyclic additive group \cite{Ru07b}; the characterization of the biskew bicyclic braces is new. Rump describes bicyclic braces via \textit{cocyclic residue classes}; our approach employs the connection between skew braces and regular subgroups of permutation groups mentioned above. We summarize this theory in Section \ref{sec_tools}, where we also show how we may reduce the problem to the classification of bicyclic braces of prime power order. We tackle this problem for powers of odd primes in Section \ref{sec_p_power}, and for powers of $ 2 $ in Section \ref{sec_2_power}, drawing on results originally developed in the context of Hopf-Galois theory. In Section \ref{sec_combining} we combine the results of the preceding sections to prove Theorem \ref{theorem_main}, and illustrate our classification with families of examples. 

\section{Tools for classifying skew braces and biskew braces} \label{sec_tools}

Our strategy has two ingredients: the correspondence between skew braces and certain regular subgroups of permutation groups, and a decomposition of a bicyclic brace as a direct product of bicyclic braces of prime-power order. In this section we summarize these ideas, mostly drawing upon \cite{GV17}, \cite{SV18}, and \cite{HAGMT}. 

A skew brace $ B=(B,+,\circ) $ naturally yields two injective group homomorphisms $ \lambda_{+}, \lambda_{\circ} : B \rightarrow \Perm(B) $ defined by $ \lambda_{+}(x)[y]=x+y $ and  $ \lambda_{\circ}(x)[y] = x \circ y $ (the left regular representations of the binary operations appearing in the skew brace). A consequence of the brace relation \eqref{eqn_brace_relation} is that the subgroup $ \lambda_{\circ}(B) $ is contained in the normalizer of $ \lambda_{+}(B) $ in $ \Perm(B) $. This is known as the \textit{holomorph} of $ (B,+) $ and is denoted $ \Hol(B,+) $; it is the semidirect product of $ \lambda_{+}(B) $ and $ \Aut(B,+) $. Conversely, given a group $ (B,+) $ and a subgroup $ M $ of $ \Hol(B,+) $ that is \textit{regular} (meaning the map $ \nu : M \rightarrow B $ defined by $ \nu(\mu) = \mu[e] $ is a bijection), we may define a binary operation $ \circ $ on $ B $ by
\[ x \circ y = \nu ( \nu^{-1}[x] \nu^{-1}[y] ); \]
we then find that $ (B,\circ) $ is a group, $ \lambda_{\circ}(B) = M $, and $ (B,+,\circ) $ is a skew brace. 

Thus given a group $ (B,+) $ the task of determining all the binary operations $ \circ $ on $ B $ such that $ (B,+,\circ) $ is a skew brace is equivalent to the task of determining all of the regular subgroups of $ \Hol(B,+) $. Since our focus is on classifying skew braces up to isomorphism, we note that two regular subgroups $ M, M' $ of $ \Hol(B,+) $ yield isomorphic skew braces if and only if $ M' = \theta M \theta^{-1} $ for some $ \theta \in \Aut(B,+) $. 

We may also characterize biskew braces from this perspective: if $ (B,+,\circ) $ is a skew brace (so that $ \lambda_{\circ}(B) $ normalizes $ \lambda_{+}(B) $ inside $ \Perm(B) $) then we may also consider $ \Hol(B,\circ) $, being the normalizer of the group $ (B,\circ) $ in $ \Perm(B) $. Reversing the roles of $ + $ and $ \circ $ in the discussion above we see that $ (B,\circ,+) $ is a skew brace (and so $ (B,+,\circ) $ is a biskew brace) if and only if $ \lambda_{+}(B) \subset \Hol(B,\circ) $. More succinctly: $ (B,+,\circ) $ is a biskew brace if and only if $ \lambda_{+}(B) $ and $ \lambda_{\circ}(B) $ mutually normalize one another inside $ \Perm(B) $. 

Next we turn to important substructures of a skew brace $ (B,+,\circ) $ called \textit{ideals}; these are the kernels of skew brace homomorphisms and the substructures that permit the formation of quotients. To give an alternative characterization we note that, since $ \lambda_{\circ}(B) \subseteq \Hol(B,+) $, for each $ x \in B $ we have $ \lambda_{\circ}(x) = \lambda_{+}(z_{x}) \gamma_{x} $ for some $ z_{x} \in B $ and $ \gamma_{x} \in \Aut(B,+) $. Evaluating each side at the identity we find that $ z_{x}=x $ for all $ x \in B $, so for all $ y \in B $ we have  $ x \circ y = x + \gamma_{x}(y) $, which implies that $ \gamma_{x}(y) = -x + (x \circ y) $. It can be shown that a subset $ A $ of $ B $ is an ideal of $ (B,+,\circ) $ if and only if $ (A,+) \trianglelefteq (B,+) $, $ (A,\circ) \trianglelefteq (B,\circ) $, and $ \gamma_{x}(A)=A $ for all $ x \in B $.

If $ B=(B,+,\circ) $ and $ B'=(B',+',\circ') $ are skew braces then the Cartesian product $ B \times B' $, with operations defined by
\begin{align}
(x,x') + (y,y') &= (x + y, \; x' +' y')  \label{align_direct_product_cdot} \\
(x,x') \circ (y,y') &= (x \circ y, \; x' \circ' y'), \label{align_direct_product_circ}
\end{align}
is a skew brace, called the \textit{external direct product} of $ B $ and $ B' $. In this case the subsets $ B \times \{ e \} $ and $ \{ e \} \times B' $ are ideals of $ B \times B' $ whose intersection is trivial and whose product with respect to either operation is equal to $ B \times B' $. We will need the following partial converse of this observation.

\begin{lemma} \label{lemma_internal_direct_product}
Let $ B = (B,+,\circ) $ be a skew brace with $ (B,+) $ and $ (B,\circ) $ abelian, and suppose that $ B $ contains ideals $ A, A' $ such that $ A \cap A' = \{ e \} $ and $ A + A' = B $. Then $ B \cong A \times A' $ as skew  braces. 
\end{lemma}
\begin{proof}
The hypotheses imply that the map $ (B,+) \rightarrow (A,+) \times (A',+) $ defined by $ x+x' \mapsto (x, \; x') $ is a well defined group isomorphism. We show that this isomorphism respects the definitions of $ \circ $ on each side. 

First let $ x \in A $ and $ x' \in A' $, and consider the quantity $ \gamma_{x}(x')-x' $. On one hand, since $ A' $ is an ideal of $ B $ we have $ \gamma_{x}(x') \in A' $, so $ \gamma_{x}(x')-x' \in A' $. On the other hand, using the definition of $ \gamma_{x} $ and the fact that $ (B,+) $ and $ (B,\circ) $ are both abelian we may write
\[ \gamma_{x}(x')-x' = (x \circ x')-x-x' = (x' \circ x)-x'-x = \gamma_{x'}(x)-x, \]
which lies in $ A $ since $ A $ is an ideal of $ B $. Hence, using the assumption that $ A \cap A' = \{ e \} $, we find that $ x + x' =  x \circ x' $. Now for $ x,y \in A $ and $ x',y' \in A' $ we have
\begin{align*}
(x+x') \circ (y+y') & = (x \circ x') \circ (y \circ y') \\
& = x \circ y \circ x' \circ y' \\
& = (x \circ y) + (x' \circ y') \\
& \mapsto (x \circ y, \; x' \circ y') \\
& = (x,x') \circ (y,y').
\end{align*}
Hence $ B \cong A \times A' $ as skew  braces.
\end{proof}

Retaining the hypotheses of Lemma \ref{lemma_internal_direct_product} and focusing on the Sylow subgroups of the additive group of a brace we obtain the following decomposition result:

\begin{lemma} \label{lemma_decomposition}
Let $ B = (B,+,\circ) $ be a skew brace with $ (B,+) $ and $ (B,\circ) $ abelian, let $ p_{1}, \ldots , p_{r} $ be the distinct prime numbers dividing $ |B| $, and for each $ i=1, \ldots ,r $ let $ (P_{i},+) $ be the unique Sylow $ p_{i} $-subgroup of $ (B,+) $. Then each $ P_{i} $ is an ideal of $ B $, and $ B $ is isomorphic as a skew brace to the direct product of the $ P_{i} $.
\end{lemma}
\begin{proof}
First let $ p $ denote one of the prime numbers $ p_{i} $ and let $ P=P_{i} $. Since $ \gamma_{x} \in \Aut(B,+) $ for each $ x \in B $, the uniqueness of $ P $ implies that $ \gamma_{x}(P)=P $ for each $ x \in B $. In particular, for $ x,y \in P $ we have $ x+\gamma_{x}(y) = x \circ y \in P $, so $ (P,\circ) \leq (B,\circ) $; since $ (B,\circ) $ is abelian this implies that $ (P,\circ) \trianglelefteq  (B,\circ) $, and so $ P $ is an ideal of the skew brace $ B $. 

Now let $ (Q,+) $ be the prime-to-$ p $ part of $ (B,+) $, so that $ Q \cap P = \{ e \} $ and $ Q + P = B $. Arguing as above, we find that $ Q $ is an ideal of the skew brace $ B $, so $ B \cong P \times Q $ as skew braces by Lemma \ref{lemma_internal_direct_product}. The result follows by induction.
\end{proof}


\section{Bicyclic braces of odd prime power order} \label{sec_p_power}

In this section we classify bicyclic braces of odd prime power order by studying the cyclic regular subgroups of the holomorph of a cyclic group of odd prime power order. These regular subgroups are studied by Kohl \cite{Ko98} and Childs \cite{TWE} in the context of Hopf-Galois theory: we give alternative proofs of their results. 

We begin with some results concerning geometric sums modulo odd prime powers; we shall refer to these repeatedly in what follows. Throughout, let $ p $ denote a fixed odd prime number.

\begin{lemma} \label{lemma_geometric_congruences}
Let $ t \in \Z $ with $ v_{p}(t-1)=m \geq 1 $. 
\begin{enumerate}
\item If $ i \in \N $ then $ 1 + t + t^{2} +\cdots + t^{p^{i}-1} \equiv p^{i} \pmod{p^{m+i}} $. \label{lemma_geometric_congruences_item_1}
\item If $ j \in \N $ then $ v_{p}(1+t+t^{2}+\cdots + t^{j-1}) = v_{p}(j) $. \label{lemma_geometric_congruences_item_2}
\end{enumerate}
\end{lemma}
\begin{proof}
Since $ v_{p}(t-1)=m $ the element $ t $ generates the subgroup of $ \Z_{p^{m+i}}^{\times} $ consisting of units that are congruent to $ 1 $ modulo $ p^{m} $, which has order $ p^{i} $. Hence the geometric sum in \ref{lemma_geometric_congruences_item_1} is congruent modulo $ p^{m+i} $ to the sum of the elements in this subgroup, which is
\[ \sum_{k=0}^{p^{i}-1} 1 + kp^{m} = p^{i} + \frac{1}{2}(p^{i}-1)p^{i}p^{m} \equiv p^{i} \pmod{p^{m+i}}. \]
This establishes \ref{lemma_geometric_congruences_item_1}. 

For \ref{lemma_geometric_congruences_item_2}, write $ j = j' p^{i} $ with $ \gcd(j',p)=1 $ (and possibly $ i=0 $). Then
\[ 1+t+t^{2}+\cdots + t^{j-1} = (1+t+t^{2}+\cdots + t^{p^{i}-1})(1+t^{p^{i}}+t^{2p^{i}}+\cdots + t^{(j'-1)p^{i}}). \]
The first factor is congruent to $ p^{i} $ modulo $ p^{m+i} $, by \ref{lemma_geometric_congruences_item_1}. The second factor consists of $ j' $ terms, each congruent to $ 1 $ modulo $ p $; since $ \gcd(j',p)=1 $ this factor is a unit modulo $ p $. Hence $ v_{p}(1+t+t^{2}+\cdots + t^{j-1}) = i = v_{p}(j) $. 
\end{proof}

Without loss of generality we identify our additive group with $ \Z_{p^{e}} $ and identify $ \Aut(\Z_{p^{e}}) $ with $ \Z_{p^{e}}^{\times} = \langle d \rangle $, where $ d $ is a fixed primitive root modulo $ p^{e} $. Hence $ \Hol(\Z_{p^{e}}) \cong \Z_{p^{e}} \rtimes \Z_{p^{e}}^{\times} $; explicitly, the operation in $ \Hol(\Z_{p^{e}}) $ is given by 
\[ (r,t)(r',t') = (r + tr', tt'). \]
Since the projection $ \Hol(\Z_{p^{e}}) \rightarrow \Z_{p^{e}}^{\times} $ is a homomorphism, an element of $ \Hol(\Z_{p^{e}})  $ has $ p $-power order only if its image in $ \Z_{p^{e}}^{\times} $ lies in the unique Sylow $ p $-subgroup of $ \Z_{p^{e}}^{\times} $; this consists of the units that are congruent to $ 1 $ modulo $ p $, has order $ p^{e-1} $, and is generated by $ q = d^{p-1} $.

\begin{proposition} \label{prop_p_e_order} 
The elements of $ \Hol(\Z_{p^{e}}) $ of order $ p^{e} $ are precisely those of the form $ (r,q^{s}) $ with $ 0 < s \leq p^{e-1} $ and $ \gcd(r,p)=1 $. 
\end{proposition}
\begin{proof}
Let $ (r,q^{s}) \in \Hol(\Z_{p^{e}}) $, and write $ t=q^{s} $. Then
\[ (r,q^{s})^{c} = (r(1+t+t^{2}+\cdots + t^{c-1}),q^{sc}), \]
which is equal to $ (0,1) $ if and only if $ p^{e-1} \mid sc $ and $ p^{e} \mid r(1+t+t^{2}+\cdots + t^{c-1}) $. By Lemma \ref{lemma_geometric_congruences} part \ref{lemma_geometric_congruences_item_2} the $ p $-valuation of the geometric sum is equal to $ v_{p}(c) $. Thus if $ \gcd(r,p)=1 $ then the minimal positive $ c $ such that $ p^{e} \mid r(1+t+t^{2}+\cdots + t^{c-1}) $ is $ c=p^{e} $; in this case we have $ p^{e-1} \mid sc $ regardless of $ s $, and so $ (r,q^{s}) $ has order $ p^{e} $. On the other hand if $ \gcd(r,p)>1 $ then choosing $ c=p^{e-1} $ we have $ p^{e} \mid r(1+t+t^{2}+\cdots + t^{c-1}) $, and $ p^{e-1} \mid sc $ regardless of $ s $, so $ (r,q^{s}) $ has order at most $ p^{e-1} $.
\end{proof}


\begin{proposition} \label{prop_p_e_regular} 
Every element of $ \Hol(\Z_{p^{e}}) $ of order $ p^{e} $ generates a regular subgroup of $ \Hol(\Z_{p^{e}}) $. 
\end{proposition}
\begin{proof}
Let $ (r,q^{s}) \in \Hol(\Z_{p^{e}}) $ have order $ p^{e} $. It is sufficient to show that the only power of $ (r,q^{s}) $ that stabilizes the identity element $ 0 \in \Z_{p^{e}} $ is the identity. Suppose that $ (r,q^{s})^{c}[0] = 0 $ for some $ c = 0, \ldots ,p^{e}-1 $. Then writing $ t=q^{s} $ as above we have
\[ (r,t)^{c}[0] = (r(1+t+t^{2}+ \cdots +t^{c-1}),t^{c}) [0] = r(1+t+t^{2}+ \cdots + t^{c-1}) \]
so $ p^{e} \mid r(1+t+t^{2}+ \cdots + t^{c-1}) $. Since $ \gcd(r,p)=1 $ by Proposition \ref{prop_p_e_order}, and $ v_{p}(1+t+t^{2}+\cdots + t^{c-1}) = v_{p}(c) $ by Lemma \ref{lemma_geometric_congruences} part \ref{lemma_geometric_congruences_item_2}, this implies $ c=0 $, as required. 
\end{proof}

\begin{corollary}
Each cyclic regular subgroup of $ \Hol(\Z_{p^{e}}) $ has a unique generator of the form $ (1,q^{s}) $ with $ 0 < s \leq p^{e-1} $. There are precisely $ p^{e-1}$ such subgroups. 
\end{corollary}
\begin{proof}
Since each cyclic subgroup of $ \Hol(\Z_{p^{e}}) $ is regular on $ \Z_{p^{e}} $ each contains a unique element of the form $ (1,q^{s}) $, which generates the subgroup since $ \gcd(1,p)=1 $. The final claim follows immediately since $ q $ has order $ p^{e-1} $ in $ \Z_{p^{e}}^{\times} $.
\end{proof}

Next we determine when two cyclic regular subgroups of $ \Hol(\Z_{p^{e}}) $ yield isomorphic bicyclic braces. By the results summarized in Section \ref{sec_tools} this occurs if and only if the subgroups are conjugate by an element of $ \Z_{p^{e}}^{\times} $.

\begin{proposition}\label{prop_iso_odd}
Two regular cyclic subgroups $ \langle (1,q^{s}) \rangle $ and $ \langle (1,q^{s'}) \rangle $ of $ \Hol(\Z_{p^{e}}) $ are conjugate by an element of $ \Z_{p^{e}}^{\times} $ if and only if $ v_{p}(s) = v_{p}(s') $. 
\end{proposition}
\begin{proof}
First suppose $ \langle (1,q^{s}) \rangle = (0,h) \langle (1,q^{s'}) \rangle (0,h)^{-1} $ for some $ h \in \Z_{p^{e}}^{\times} $. Projecting to the $ \Z_{p^{e}}^{\times}  $ component we have $ \langle q^{s} \rangle = h \langle q^{s'} \rangle h^{-1} $ in $ \Z_{p^{e}}^{\times} $, so in fact we have $ \langle q^{s} \rangle = \langle q^{s'} \rangle $. Hence $ \gcd(s,p^{e-1})=\gcd(s',p^{e-1}) $, and so $ v_{p}(s)=v_{p}(s') $. 

Conversely, suppose that $ v_{p}(s) = v_{p}(s') $. Then $ s' \equiv c s \pmod{p^{e-1}} $ for some $ c $ with $ (c,p)=1 $. Writing $ t = q^{s} $ and $ t'=q^{s'} $ we have 
\[ (1,q^{s})^{c} = (1+t+ \cdots + t^{c-1},t^{c}) = (h,t') \mbox{ say}  \]
and $ h=1+t+ \cdots +t^{c-1} \in \Z_{p^{e}}^{\times} $, by Lemma \ref{lemma_geometric_congruences} part \ref{lemma_geometric_congruences_item_2}. Hence we have
\[ (0,h)^{-1} (1,t)^{c} (0,h) = (0,h)^{-1}(h, t')(0,h) = (1,t'), \]
and so $\langle(1,q^{s})\rangle$ and $\langle(1,q^{s'})\rangle$ are conjugate by an element of $ \Z_{p^{e}}^{\times} $.
\end{proof}

\begin{corollary}
There are precisely $ e $ isomorphically distinct bicyclic braces of odd prime power order $ p^{e} $. They correspond bijectively with the subgroups $ \langle (1,q^{p^{k}}) \rangle $ with $ 0 \leq k \leq e-1 $. 
\end{corollary}


Finally, we give an explicit representative of each isomorphism class of bicyclic brace of order $ p^{e} $. As described in Section \ref{sec_tools} the map $ \nu : \langle (1,q^{p^{k}}) \rangle \rightarrow \Z_{p^{e}} $ defined by $ \nu(\mu) = \mu[0] $ is a bijection, and the corresponding circle operation on $ \Z_{p^{e}} $ is given by 
\[ i \circ j = \nu ( \nu^{-1}[i] \nu^{-1}[j] ). \]

\begin{proposition} \label{prop_circle_p}
Let $ 0 \leq k \leq e-1 $ and let $ t = q^{p^{k}} $. Then the bicyclic brace corresponding to the subgroup $ \langle (1,t) \rangle $ of $ \Hol(\Z_{p^{e}}) $ is $ (\Z_{p^{e}},+,\circ) $, where
\[ i \circ j = i+j+(t-1)ij. \]
\end{proposition}
\begin{proof}
Let $ 0 \leq i', j' \leq p^{e}-1 $ be the unique values such that $ (1,t)^{i'}[0]=i $ and $ (1,t)^{j'}[0]=j $. Then in particular $ 1+t+\cdots + t^{i'-1} \equiv \frac{t^{i'}-1}{t-1} \equiv i \pmod{p^{e}} $, and we have
\begin{eqnarray*}
i \circ j & = & (1,t)^{i'}(1,t)^{j'}[0] \\
& = & (1,t)^{i'}[j] \\
& = & i+ t^{i'}j \\
& = & i+((t-1)i+1)j \\
& = & i+j+(t-1)ij. 
\end{eqnarray*}
\end{proof}

Since $ q $ generates the subgroup of units of $ \Z_{p^{e}}^{\times} $ that are congruent to $ 1 $ modulo $ p $ we have $ q^{p^{k}} \equiv 1 \pmod{p^{k+1}} $; hence we obtain

\begin{corollary} \label{cor_circle_p}
The $ e $ isomorphically distinct bicyclic braces of order $ p^{e} $ are represented by $ (\Z_{p^{e}},+,\circ_{k}) $ where $ 0 \leq k \leq e-1 $ and
\[ i \circ_{k} j = i + j + w_{k}ij \mbox{ with } v_{p}(w_{k})=k+1. \]
\end{corollary}

\section{Bicyclic braces of $ 2 $-power order} \label{sec_2_power}

In this section we classify bicyclic braces of $ 2 $-power order. Our approach is essentially the same as that taken in Section \ref{sec_p_power}, with some modifications to account for the fact that the automorphism group of a cyclic $ 2 $-group is in general not cyclic. Once again we are able to employ results developed in the context of Hopf-Galois theory; in this case by Byott \cite{By07}. In particular, from \cite[Proposition 7.1]{By07} we have the following variant of Lemma \ref{lemma_geometric_congruences}:

\begin{lemma} \label{lemma_geometric_congruences_2}
Let $ t \in \Z $ with $ v_{2}(t-1)=m \geq 1 $. 
\begin{enumerate}
\item If $ m \geq 2 $ and $ i \in \N $ then $ 1 + t + t^{2} +\cdots + t^{2^{i}-1} \equiv 2^{i} - 2^{i+m-1} \pmod{2^{m+i}} $. \label{lemma_geometric_congruences_2_item_1}
\item If $ j \in \N $ then 
\[ v_{2}(1+t+t^{2}+\cdots + t^{j-1}) = 
\left\{ \begin{array}{ll} v_{2}(j) + v_{2}(t+1) - 1 & \mbox{ if $ m=1 $ and $ j $ is even} \\
v_{2}(j) & \mbox{ otherwise}. 
\end{array} \right. \]
\label{lemma_geometric_congruences_2_item_2}
\end{enumerate}
\end{lemma}


Now as in Section \ref{sec_p_power} we consider the group $ \Z_{2^{e}} $. If $ e=1 $ or $ e=2 $ then up to isomorphism there is a unique bicyclic brace of order $ 2^{e} $ (see \cite[Proposition 2.4]{Ba15}); this is $ (\Z_{2^{e}},+,+) $, and of course is biskew. We therefore assume that $ e \geq 3 $ and identify $ \Aut(\Z_{2^{e}}) $ with $ \Z_{2^{e}}^{\times} = \langle -1 \rangle \times \langle 5 \rangle $; hence $ \Hol(\Z_{2^{e}}) \cong \Z_{2^{e}} \rtimes \Z_{2^{e}}^{\times} $. As before we begin by identifying the elements of $ \Hol(\Z_{2^{e}}) $ of order $ 2^{e} $. 

\begin{proposition} \label{prop_2_e_order} 
The elements of $ \Hol(\Z_{2^{e}}) $ of order $ 2^{e} $ are precisely those of the form $ (r,5^{s}) $ with $ 0 < s \leq 2^{e-2} $ and $ r $ odd. 
\end{proposition}
\begin{proof}
Let $ (r,t) \in \Hol(\Z_{p^{e}}) $; then
\[ (r,t)^{c} = (r(1+t+t^{2}+\cdots + t^{c-1}),t^{c}). \]

If $ t \equiv -1 \pmod{4} $ then by Lemma \ref{lemma_geometric_congruences_2} part \ref{lemma_geometric_congruences_2_item_2} we see that $ v_{2}(1+t+t^{2}+\cdots + t^{2^{e-1}-1})=e $; since $ t $ has order at most $ 2^{e-2} $ we see that $ (r,t)^{2^{e-1}} = (0,1) $, so $ (r,t) $ does not have order $ 2^{e} $. 

If $ t \equiv 1 \pmod{4} $, so that $ t \equiv 5^{s} $ for some $ s $, then $ v_{2}(1+t+t^{2}+\cdots + t^{c-1})=v_{2}(c) $. Now if $ r $ is odd then the minimal positive $ c $ such that $ 2^{e} \mid r(1+t+t^{2}+\cdots + t^{c-1}) $ is $ c=2^{e} $, and so $ (r,t) $ has order $ 2^{e} $. On the other hand if $ r $ is even then $ 2^{e} \mid r(1+t+t^{2}+\cdots + t^{2^{e-1}-1}) $, so $ (r,t) $ has order at most $ 2^{e-1} $.
\end{proof}


\begin{proposition} \label{prop_2_e_regular} 
Every element of $ \Hol(\Z_{2^{e}}) $ of order $ 2^{e} $ generates a regular subgroup of $ \Hol(\Z_{2^{e}}) $. 
\end{proposition}
\begin{proof}
This is essentially the same as the proof of Proposition \ref{prop_p_e_regular}. 
\end{proof}

\begin{corollary}
Each cyclic regular subgroup of $ \Hol(\Z_{2^{e}}) $ has a unique generator of the form $ (1,5^{s}) $ with $ 0 < s \leq 2^{e-2} $ . There are precisely $ 2^{e-2} $ such subgroups. 
\end{corollary}

Next we determine when two cyclic regular subgroups yield isomorphic braces. 

\begin{proposition}
Two regular cyclic subgroups $ \langle (1,5^{s}) \rangle $ and $ \langle (1,5^{s'}) \rangle $ of $ \Hol(\Z_{2^{e}}) $ are conjugate by an element of $ \Z_{2^{e}}^{\times} $ if and only if $ v_{2}(s) = v_{2}(s') $. 
\end{proposition}
\begin{proof}
Although $ \Z_{2^{e}}^{\times} $ is not cyclic, it is certainly abelian; hence the argument employed to prove the first half of  Proposition \ref{prop_iso_odd} shows that if $ \langle (1,5^{s}) \rangle $ and $ \langle (1,5^{s'}) \rangle $ are conjugate by an element of $ \Z_{2^{e}}^{\times} $ then $ v_{2}(s) = v_{2}(s') $. 

For the converse, note that if $ v_{2}(s) = v_{2}(s') $ then $ s' \equiv cs \pmod{2^{e-2}} $ for some odd $ c $; hence, writing $ t=5^{s} $, we have $ v_{2}(1+t+t^{2}+ \cdots +t^{c-1}) = v_{2}(c)= 0 $ by Lemma \ref{lemma_geometric_congruences_2} part (ii). Now the argument employed to prove the second half of Proposition \ref{prop_iso_odd} applies. 
\end{proof}

\begin{corollary}
There are precisely $ e-1 $ isomorphically distinct bicyclic braces of order $ 2^{e} $. 
\end{corollary}

Finally, we give an explicit representative of each isomorphism class of bicyclic brace of order $ 2^{e} $. 

\begin{proposition} \label{prop_circle_2}
Let $ 0 \leq k \leq e-2 $. Let $ t = 5^{2^{k}} $. Then the bicyclic brace corresponding to the subgroup $ \langle (1,t) \rangle $ of $ \Hol(\Z_{2^{e}}) $ is $ (\Z_{2^{e}},+,\circ) $, where
\[ i \circ j = i+j+(t-1)ij. \]
\end{proposition}
\begin{proof}
This is essentially the same as the proof of Proposition \ref{prop_circle_p}.
\end{proof}

Since the element $ 5 $ generates the group of units of $ \Z_{2^{e}}^{\times} $ that are congruent to $ 1 $ modulo $ 4 $ we obtain the following corollary. We take this opportunity to integrate the cases $ e=1 $ and $ e=2 $ into our discussion. 

\begin{corollary} \label{cor_circle_2}
There is a unique bicyclic brace of order $ 2 $, which is represented by $ (\Z_{2},+,+) $. If $ e \geq 2 $ then the $ e-1 $ isomorphically distinct bicyclic braces of order $ 2^{e} $ are represented by $ (\Z_{2^{e}},+,\circ_{k}) $ where $ 0 \leq k \leq e-2 $ and 
\[ i \circ_{k} j = i + j + w_{k}ij \mbox{ with } v_{2}(w_{k}) = k+2. \]
\end{corollary}



\section{Bicyclic braces of arbitrary order and biskew braces} \label{sec_combining}

In this section we complete the proof of Theorem \ref{theorem_main} by using Lemma \ref{lemma_decomposition} to combine the results of Sections \ref{sec_p_power} and \ref{sec_2_power} to classify bicyclic braces of arbitrary order and using the theory discussed in Section \ref{sec_tools} to identify which of these are biskew. We conclude with some familes of examples. 

\begin{proposition} \label{prop_circle_general}
Let $ n \in \N $. The isomorphically distinct bicyclic braces of order $ n $ may be realized as $ (\Z_{n},+,\circ_{w}) $ where
\begin{equation} \label{eqn_circle}
i \circ_{w} j = i + j + wij
\end{equation}
and $ w $ is a positive divisor of $ n $ such that $ p \mid w $ for all $ p \mid n $ and $ 4 \mid w $ if $ 4 \mid n $. 
\end{proposition}
\begin{proof}
By Lemma \ref{lemma_decomposition} a bicyclic brace with additive group $ (\Z_{n},+) $ is isomorphic to the direct product of bicyclic braces whose additive groups are the Sylow subgroups of $ (\Z_{n},+) $. Applying Corollary \ref{cor_circle_p} to each component of odd prime power order and Corollary \ref{cor_circle_2} to the component of $ 2 $-power order, and recalling the formula \eqref{align_direct_product_circ} for the circle operation in a direct product of skew braces, we see that the isomorphically distinct bicyclic braces of order $ n $ correspond to the divisors of $ n $ with the properties given and that the circle operations in each component combine to give \eqref{eqn_circle}. 
\end{proof}


\begin{proposition} \label{prop_biskew_general}
Let $ n \in \N $ and $ w $ be a divisor of $ n $ such that the binary operation $ \circ $ defined by \eqref{eqn_circle} makes $ (\Z_{n},+,\circ) $ into a bicyclic brace. Then $ (\Z_{n},+,\circ) $ is a biskew brace if and only if $ n \mid w^{2} $.  
\end{proposition}
\begin{proof}
Recall from Section \ref{sec_tools} that $ (\Z_{n},+,\circ) $ is a biskew brace if and only if the image of the left regular representation $ \lambda_{+} : \Z_{n} \rightarrow \Perm(\Z_{n}) $ normalizes the image of the left regular representation $ \lambda_{\circ} : \Z_{n} \rightarrow \Perm(\Z_{n}) $. Since $ (\Z_{n},+) $ and $ (\Z_{n},\circ) $ are both cyclic and generated by $ 1 $, this occurs if and only if $ \lambda_{+}(1) \lambda_{\circ}(1) \lambda_{+}(1)^{-1} \in \lambda_{\circ}(\Z_{n}) $. That is, if and only if 
\begin{equation} \label{eqn_biskew_condition}
\lambda_{+}(1) \lambda_{\circ}(1) \lambda_{+}(-1) = \lambda_{\circ}(j) \mbox{ for some } j.
\end{equation}
For $ i \in \Z_{n} $ we have
\begin{eqnarray*}
\lambda_{+}(1) \lambda_{\circ}(1) \lambda_{+}(-1) [i] & = & 1+( 1 \circ (i-1) ) \\
& = & 1 + 1+(i-1)+w(i-1) \\
& = & 1+i+w(i-1),
\end{eqnarray*}
whereas $ \lambda_{\circ}(j)[i] = i+j+wij $. Thus \eqref{eqn_biskew_condition} holds if and only if 
\[ 1+i+w(i-1) \equiv i+j+wij \pmod{n} \mbox{ for all } i. \]
Choosing $ i=0 $ we see that $ j=1-w $, so the condition becomes 
\[ 1+i+w(i-1) \equiv i+1-w+wi(1-w) \pmod{n} \mbox{ for all } i; \]
that is, $ w^{2} \equiv 0 \pmod{n} $. 
\end{proof}

\begin{example}
Suppose that $ n $ is odd and squarefree. Then by Proposition \ref{prop_circle_general} the only value of $ w $ for which $ (\Z_{n},+,\circ_{w}) $ is a bicyclic brace is $ w=n $. This yields the trivial skew brace $ (\Z_{n},+,+) $, which of course if biskew. 

Conversely, if $ n $ is odd and not squarefree then let $ p $ be a prime number such that $ p^{2} \mid n $ and let $ w = n/p $. Then $ (\Z_{n},+,\circ_{w}) $ is a nontrivial bicyclic brace. Moreover, since $ n \mid w^{2} $, this brace is biskew. 
\end{example}

\begin{example}
Suppose that $ n = (p_{1} \ldots p_{r})^{e} $ for distinct odd prime numbers $ p_{1}, \ldots ,p_{r} $. Then by Proposition \ref{prop_circle_general} the values of $ w $ for which $ (\Z_{n},+,\circ_{w}) $ is a bicyclic brace are precisely the numbers $ p_{1}^{e_{1}} \ldots p_{r}^{e_{r}} $ with $ 1 \leq e_{i} \leq e $ for each $ i $, and a bicylic brace of this form is biskew if and only if $ e \leq 2e_{i} $ for each $ i $. Hence there are $ e^{r} $ isomorphically distinct bicyclic braces in this case, of which $ \lfloor (e+2)/2 \rfloor ^{r} $ are biskew. 
\end{example}

\begin{example}
Let $ n = 7500 = 2^2 \cdot 3^1 \cdot 5^4$. By Proposition \ref{prop_circle_general} the isomorphically distinct bicyclic braces of order $ 7500 $ are given by $ (\Z_{n},+,\circ_{w}) $ with $ w=60, 300, 1500, $ or $ 7500 $, with the final three of these being biskew. 

If (for example) $ w = 10 $ then $ \circ_{w} $ does not give a group structure on $ \Z_{n} $: for example, the element $ 5000 $ is idempotent with respect to $ \circ_{10} $. 

If $ w = 60 $ then $ (\Z_{n},+,\circ_{w}) $ is a bicyclic brace, but (for example) 
\[ 1+(2 \circ 3) = 1 + 5 + (2 \cdot 3 \cdot 60) = 366, \]
whereas 
\[  (1 + 2) \circ \bar{1} \circ (1+3) = 3 \circ 2459 \circ 4 = 7266, \]
so this brace is not biskew. 
\end{example}

\bibliography{Omahabib}
\bibliographystyle{plain}

\end{document}